\documentclass[amssymb, 11pt]{amsart}
\usepackage{amsmath}
\usepackage{amsthm}
\usepackage{amsfonts}

\newtheorem{thm}{Theorem}[section]
\newtheorem{theorem}[thm]{Theorem}

\newtheorem{lemma}[thm]{Lemma}

\newcommand{\ee}{\mathbb{E}}
\newcommand{\pp}{\mathbb{P}}

\newcommand{\cc}{\mathbb{C}}

\newcommand{\rr}{\mathbb{R}}

\begin{document}

\title [Stein's Method and Characters of Compact Lie Groups]
{Stein's Method and Characters of Compact Lie Groups}

\author{Jason Fulman}
\address{Department of Mathematics\\
University of Southern California\\
Los Angeles, CA 90089-2532} \email{fulman@usc.edu}



\thanks{The author received funding from NSF grant DMS-0503901.}

\date{Submitted 6/13/08; minor revisions on 6/22/08}

\begin{abstract} Stein's method is used to study the trace of a random
element from a compact Lie group or symmetric space. Central limit
theorems are proved using very little information: character values on a
single element and the decomposition of the square of the trace into
irreducible components. This is illustrated for Lie groups of classical
type and Dyson's circular ensembles. The approach in this paper will be
useful for the study of higher dimensional characters, where normal
approximations need not hold.
\end{abstract}

\maketitle

\section{Introduction} \label{introduction}

There is a large literature on the traces of random elements of compact
Lie groups. One of the earliest results is due to Diaconis and Shahshahani
\cite{DS}. Using the method of moments, they show that if $g$ is random
from the Haar measure of the unitary group $U(n,\cc)$, and $Z=X+iY$ is a
standard complex normal with $X$ and $Y$ independent, mean $0$ and
variance $\frac{1}{2}$ normal variables, then for $j=1,2,\cdots$,
$Tr(g^j)$ are independent and distributed as $\sqrt{j}Z$ asymptotically as
$n \rightarrow \infty$. They give similar results for the orthogonal group
$O(n,\rr)$ and the group of unitary symplectic matrices $USp(2n,\cc)$.
Concerning the error in the normal approximation, Diaconis conjectured
that for fixed $j$, it decreases exponentially or even subexponentially in
$n$. Stein \cite{St2} uses ``Stein's method'' to show that $Tr(g^k)$ on
$O(n,\rr)$ is asymptotically normal with error $O(n^{-r})$ for any fixed
$r$. Johansson \cite{J} proved Diaconis' conjecture for classical compact
Lie groups using Toeplitz determinants and a very detailed analysis of
characteristic functions.

One direction in which the results of the previous paragraph have been
extended is the study of linear statistics of eigenvalues: see \cite{J},
\cite{DE}, \cite{So} and the numerous references therein. There is also
work by D'Aristotile, Diaconis, and Newman \cite{DDN} on central limit
theorems for linear functions such as $Tr(Ag)$ where $A$ is a fixed $n
\times n$ real matrix and $g$ is from the Haar measure of $O(n,\rr)$. In
recent work, Meckes \cite{Me2} refined Stein's technique from \cite{St2}
to establish a sharp total variation distance error term (order $n^{-1}$)
for the \cite{DDN} result.

A natural goal is to prove limit theorems (with error terms) for the
distribution of traces in other irreducible representations: i.e.
$\chi^{\tau}(g)$, where $g$ is a random element of a compact Lie group and
$\chi^{\tau}$ is the character of an irreducible representation $\tau$.
This would have direct implications for Katz's work \cite{Ka} on
exponential sums; see Section 4.7 of \cite{KLR} for details. We do not
attain this goal, but make a useful contribution to it.

More precisely, the current paper presents a formulation of Stein's method
designed for the study of $\chi^{\tau}(g)$. In the case of normal
approximation, we obtain $O(n^{-1})$ bounds for the error term using only
two pieces of information:

\begin{itemize}
\item The value of the ``character ratios'' $\frac{\chi^{\phi}(\alpha)}
{dim(\phi)}$ where $\phi$ may be arbitrary but $\alpha$ is a single
element of $G$ (typically chosen to be close to the identity)

\item The decomposition of $\tau^2$ into irreducible representations
\end{itemize}

In contrast, the method of moments approach requires knowing the
multiplicity of the trivial representation in $\tau^k$ for all $k \geq 1$
(which could be tricky to compute) and does not give an immediate bound on
the error. Johansson's paper \cite{J} gives sharper bounds when
$\chi^{\tau}$ is the trace of an element from a classical compact Lie
group, but required knowledge of high order moments and deep analytical
tools which might not extend to arbitrary representations $\tau$. Even the
Stein's method approaches of Stein \cite{St2} and Meckes \cite{Me2} use
information about the distribution of matrix entries; very little is known
about this for arbitrary $\tau$, whereas the main ingredient for our
approach (character theory) is well-developed.

Let us explain our statement in the abstract that the methods of this
paper will prove useful for approximation other than normal approximation.
We use Stein's method of exchangeable pairs which involves the
construction of a pair $(W,W')$ of exchangeable random variables. Our pair
(which is somewhat different from those of Stein \cite{St2} and Meckes
\cite{Me2}) satisfies the linearity condition that $\ee(W'|W)$ is
proportional to $W$, and we find representation theoretic formulas for
quantities such as $\ee(W'-W)^k$. These computations are completely
general and apply to arbitrary distributional approximation. Stein's
method of exchangeable pairs is still quite undeveloped for continuous
distributions other than the normal, but that is temporary and as evidence
we mention the paper \cite{CF} which develops error terms for exponential
approximation using quantities like $\ee(W'-W)^k$ with $k$ small.

We remark that the bounds in our paper are all given in the Kolmogorov
metric. Similar results can be proved in the slightly stronger total
variation metric (see the remarks after Theorem \ref{steinbound}). However
we prefer to work in the Kolmogorov metric as it underscores the
similarity with discrete settings such as \cite{Fu}, where total variation
convergence does not occur. We also mention that all bounds obtained in
this paper are given with explicit constants.

The organization of this paper is as follows. Section \ref{stein} gives
background on Stein's method and normal approximation. Section \ref{group}
develops general theory for the case that $G$ is a compact Lie group and
$\chi^{\tau}$ an irreducible character. It treats the trace of random
elements of $O(n,\rr)$, $USp(2n,\cc)$, and $U(n,\cc)$ as examples. Section
\ref{sspace} extends the methods of Section \ref{group} to study spherical
functions of compact symmetric spaces. The symmetric space setting is
natural from the viewpoint of random matrix theory \cite{Dn}, \cite{KaS}.
After illustrating the technique on the sphere, we treat Dyson's circular
ensembles as examples, obtaining an error term.

\section{Stein's Method for Normal Approximation} \label{stein}

    In this section we briefly review Stein's method for normal
    approximation, using the method of exchangeable pairs
    \cite{St1}. For more details, one can consult the survey \cite{RR}
    and the references therein.

    Two random variables $W,W'$ are called an exchangeable pair if
    $(W,W')$ has the same distribution as $(W',W)$. As is typical in probability
    theory, let $\ee(A|B)$ denote the expected value of $A$
    given $B$. The following result of Stein uses an exchangeable
    pair $(W,W')$ to prove a central limit theorem for $W$.

\begin{theorem} \label{steinbound} (\cite{St1}) Let $(W,W')$ be an
 exchangeable pair of real random variables such that $\ee(W^2)=1$ and
$\ee(W'|W) = (1-a)W$ with $0< a <1$. Then for all real $x_0$,
\begin{eqnarray*} & & \left| \pp(W \leq x_0) - \frac{1}{\sqrt{2 \pi}}
\int_{-\infty}^{x_0} e^{-\frac{x^2}{2}} dx \right|\\ & \leq &
\frac{\sqrt{Var(\ee[(W'-W)^2|W])}}{a} + (2 \pi)^{- \frac{1}{4}}
\sqrt{\frac{1}{a} \ee|W'-W|^3}. \end{eqnarray*}
\end{theorem}

{\it Remarks}:
\begin{enumerate}

\item There are variations of Theorem \ref{steinbound} (for instance
Theorem 6 of \cite{Me1}) which can be combined with our calculations to
prove normal approximation in the total variation metric. However Theorem
\ref{steinbound} is quite convenient for our purposes.

\item In recent work, R\"{o}ellin \cite{Rl} has given a version of Theorem
\ref{steinbound} in which the exchangeability condition can be replaced by
the slightly weaker condition that $W$ and $W'$ have the same law. Since
exchangeability holds in our examples and may be useful for other
applications involving Stein's method, we adhere to using Theorem
\ref{steinbound}.

\end{enumerate}

    To apply Theorem \ref{steinbound}, one needs bounds on $Var(\ee[(W'-W)^2|W])$
    and $\ee|W-W'|^3$. The following lemmas are helpful for this purpose.

\begin{lemma} \label{var} Let $(W,W')$ be an exchangeable pair of random
variables such that $\ee(W'|W) = (1-a) W$ and $\ee(W^2)=1$. Then
$\ee(W'-W)^2 = 2a$.
\end{lemma}

\begin{proof} Since $W$ and $W'$ have the same distribution, \begin{eqnarray*}
 \ee(W'-W)^2 & = & \ee(\ee(W'-W)^2|W)\\ & = & \ee((W')^2) + \ee(W^2) -
2 \ee(W \ee(W'|W))\\ & = & 2 \ee(W^2) - 2 \ee(W \ee(W'|W))\\ & = & 2
 \ee(W^2) - 2(1-a) \ee(W^2)\\ & = & 2a. \end{eqnarray*} \end{proof}

    Lemma \ref{majorize} is a well known inequality (already used
    in the monograph \cite{St1}) and useful because often the
    right hand side is easier to compute or bound than the left
    hand side. To make this paper as self-contained as possible,
    we include a proof. Here $x$ is an element of the state space
    $X$.

\begin{lemma} \label{majorize} \[ Var(\ee[(W'-W)^2|W]) \leq Var(\ee[(W'-W)^2|x]).\]
\end{lemma}

\begin{proof} Jensen's
 inequality states that if $g$ is a convex function, and $Z$ a random
 variable, then $g(\ee(Z)) \leq \ee(g(Z))$. There is also a
 conditional version of Jensen's inequality (Section 4.1 of \cite{Du})
 which states that for any $\sigma$ subalgebra ${\it F}$ of the
 $\sigma$-algebra of all subsets of $X$, \[ \ee(g(\ee(Z|{\it F})))
 \leq \ee(g(Z)).\] The lemma follows by setting $g(t)=t^2$,
 $Z=\ee((W'-W)^2|x)$, and letting ${\it F}$ be the $\sigma$-algebra
 generated by the level sets of $W$. \end{proof}

\section{Compact Lie groups} \label{group}

    This section uses Stein's method to study the distribution of
    a fixed irreducible character $\chi^{\tau}$ of a compact Lie group $G$. Subsection
    \ref{gengroup1} develops general theory for the case that $\chi^{\tau}$ is real
    valued. This is applied to
    study the trace of a random element of $USp(2n,\cc)$ in Subsection
    \ref{Sp} and the trace of a random orthogonal matrix in Subsections \ref{SO}
    and \ref{O}. Subsection \ref{cgengroup1} indicates the relevant amendments for the
    complex setting and Subsection \ref{U} illustrates the theory for $U(n,\cc)$.

\subsection{General theory (real case)} \label{gengroup1}

    Let $G$ be a compact Lie group and $\chi^{\tau}$ a non-trivial
    real-valued irreducible character of $G$. The random variable of interest
    to us is $W= \chi^{\tau}(g)$, where $g$ is chosen from the Haar measure of
    $G$. It follows from the orthogonality relations for
    irreducible characters of $G$ that $\ee(W)=0$ and
    $\ee(W^2)=1$.

    The following functional equation will be useful.

\begin{lemma} \label{dens1} (\cite{He2}, p. 392) Let $G$ be a compact Lie group
and $\chi^{\phi}$ an irreducible character of $G$. Then \[ \int_G
\chi^{\phi}(h \alpha h^{-1}g) dh = \frac{\chi^{\phi}(\alpha)}{dim(\phi)}
\chi^{\phi}(g) \] for all $\alpha,g \in G$.
\end{lemma}

    We now define a pair $(W,W')$ by letting $W=\chi^{\tau}(g)$ where $g$ is
    chosen from Haar measure and $W'=W(\alpha g)$ where $\alpha$
    is chosen uniformly at random from a fixed self-inverse conjugacy class of
    $G$. Exchangeability of $(W,W')$ follows since the conjugacy class of
    $\alpha$ is self-inverse. Moreover, since
    $\chi^{\phi}(\alpha^{-1})=\overline{\chi^{\phi}(\alpha)}$, one has
    that $\chi^{\phi}(\alpha)$ is real for all irreducible representations $\phi$,
    a fact which will be used freely throughout this subsection.

    The remaining results in this subsection show that the
    exchangeable pair $(W,W')$ has desirable properties.

\begin{lemma} \label{mom1} \[ \ee(W'|W) = \left( \frac{\chi^{\tau}(\alpha)}
{dim(\tau)} \right) W.\] \end{lemma}

\begin{proof} Applying Lemma \ref{dens1} with $\phi=\tau$, one has that
\[ \ee(W'|g) = \int_{h \in G} \chi^{\tau}(h \alpha h^{-1} g) dh =
\left( \frac{\chi^{\tau}(\alpha)}{dim(\tau)} \right) \chi^{\tau}(g).\] The
result follows since this depends on $g$ only through $W$.
\end{proof}

\begin{lemma} \label{mom2} \[ \ee(W'-W)^2 = 2 \left(1-  \frac{\chi^{\tau}(\alpha)}
{dim(\tau)} \right).\] \end{lemma}

\begin{proof} This is immediate from Lemmas \ref{var} and \ref{mom1}. \end{proof}

    For the remainder of this subsection, if $\phi$ is an
    irreducible representation of $G$, we let $m_{\phi}(\tau^r)$
    denote the multiplicity of $\phi$ in the r-fold tensor product
    of $\tau$ (which has character $(\chi^{\tau})^r$).

\begin{lemma} \label{cond2} \[ \ee[(W')^2|g] = \sum_{\phi} m_{\phi}(\tau^2)
\frac{\chi^{\phi}(\alpha)}{dim(\phi)} \chi^{\phi}(g), \] where the sum is
over all irreducible representations of $G$.
\end{lemma}

\begin{proof} Write $(W')^2= \sum_{\phi} m_{\phi}(\tau^2)
\chi^{\phi}(g')$. Lemma \ref{dens1} gives that
\[ \ee[ \chi^{\phi}(g')|g] = \int_G \chi^{\phi}(h \alpha h^{-1}g) dh =
\frac{\chi^{\phi}(\alpha)}{dim(\phi)} \chi^{\phi}(g),\] and the result
follows. \end{proof}

    Lemma \ref{condvar} writes $Var([\ee(W'-W)^2|g])$ as a sum of
    positive quantities.

\begin{lemma} \label{condvar} \[ Var([\ee(W'-W)^2|g]) = \sum_{\phi}^{\ \ *}
 m_{\phi}(\tau^2)^2 \left( 1 + \frac{\chi^{\phi}(\alpha)}{dim(\phi)}
  - \frac{2 \chi^{\tau}(\alpha)}{dim(\tau)} \right)^2, \]
 where the star signifies that the sum is over all nontrivial
 irreducible representations of $G$.
\end{lemma}

\begin{proof} By Lemmas \ref{mom1} and \ref{cond2}, \begin{eqnarray*}
 \ee((W'-W)^2|g) & = & \ee[(W')^2|g] -2W \ee(W'|g) +W^2\\
& = & \ee[(W')^2|g] + \left(1-\frac{2 \chi^{\tau}(\alpha)}{dim(\tau)}
\right) W^2\\
 & = & \sum_{\phi} m_{\phi}(\tau^2) \left( 1 +
\frac{\chi^{\phi}(\alpha)}{dim(\phi)} - \frac{2
\chi^{\tau}(\alpha)}{dim(\tau)} \right)
\chi^{\phi}(g). \end{eqnarray*} The orthogonality relation for
irreducible characters of $G$ gives that \[ \ee[ \ee((W'-W)^2|g)^2 ] =
\sum_{\phi} m_{\phi}(\tau^2)^2 \left( 1 +
\frac{\chi^{\phi}(\alpha)}{dim(\phi)} - \frac{2
\chi^{\tau}(\alpha)}{dim(\tau)} \right)^2 .\] Finally, note that
\[  Var([\ee(W'-W)^2|g]) =  \ee[ \ee((W'-W)^2|g)^2 ] - (\ee(W'-W)^2)^2, \]
and since the multiplicity of the trivial representation in $\tau^2$ is
$1$, the result follows from Lemma \ref{mom2}.
\end{proof}

\begin{lemma} \label{highmom} Let $k$ be a positive integer.
\begin{enumerate}
\item $\ee(W'-W)^k = \sum_{r=0}^k (-1)^{k-r} {k \choose r} \sum_{\phi}
m_{\phi}(\tau^r) m_{\phi}(\tau^{k-r})
\frac{\chi^{\phi}(\alpha)}{dim(\phi)}$.
\item $\ee(W'-W)^4 = \sum_{\phi} m_{\phi}(\tau^2)^2
\left[ 8 \left(1 - \frac{\chi^{\tau}(\alpha)}{dim(\alpha)} \right) - 6
\left(1-\frac{\chi^{\phi}(\alpha)}{dim(\phi)} \right) \right]$.
\end{enumerate}
\end{lemma}

\begin{proof} For the first assertion, note that \[ \ee[(W'-W)^k|g] =
\sum_{r=0}^k (-1)^{k-r} {k \choose r} \chi^{\tau}(g)^{k-r}
\ee[(W')^r|g].\] Arguing as in Lemma \ref{cond2} gives that this is equal
to \[ \sum_{r=0}^k (-1)^{k-r} {k \choose r} \chi^{\tau}(g)^{k-r}
\sum_{\phi} m_{\phi}(\tau^r) \frac{\chi^{\phi}(\alpha)}{dim(\phi)}
\chi^{\phi}(g).\] Thus
$\ee(W'-W)^k$ is equal to \begin{eqnarray*} & & \ee(\ee[(W'-W)^k|g])\\
& = & \sum_{r=0}^k (-1)^{k-r} {k \choose r} \sum_{\phi} m_{\phi}(\tau^r)
\frac{\chi^{\phi}(\alpha)}{dim(\phi)} \int_{g \in G} \chi^{\tau}(g)^{k-r}
\chi^{\phi}(g)\\ & = & \sum_{r=0}^k (-1)^{k-r} {k \choose r} \sum_{\phi}
m_{\phi}(\tau^r) m_{\phi}(\tau^{k-r})
\frac{\chi^{\phi}(\alpha)}{dim(\phi)}. \end{eqnarray*}

    For the second assertion, note by the first assertion that \[
\ee(W'-W)^4 = \sum_{r=0}^4 (-1)^{r} {4 \choose r} \sum_{\phi}
m_{\phi}(\tau^r) m_{\phi}(\tau^{4-r})
\frac{\chi^{\phi}(\alpha)}{dim(\phi)}.\] If $\alpha$ is the identity
element of $G$, then $W'=W$ which implies that \[ 0 = \sum_{r=0}^4
(-1)^{r} {4 \choose r} \sum_{\phi} m_{\phi}(\tau^r) m_{\phi}(\tau^{4-r})
.\] Thus for general $\alpha$,
\[ \ee(W'-W)^4 = - \sum_{r=0}^4 (-1)^{r} {4 \choose r} \sum_{\phi} m_{\phi}(\tau^r)
m_{\phi}(\tau^{4-r}) \left( 1 - \frac{\chi^{\phi}(\alpha)}{dim(\phi)}
\right).\] Observe that the $r=0,4$ terms in this sum vanish, since the
only contribution could come from the trivial representation, which
contributes 0. The $r=2$ term is \[ -6 \sum_{\phi}
\left(1-\frac{\chi^{\phi}(\alpha)}{dim(\phi)} \right)
m_{\phi}(\tau^2)^2.\] The $r=1,3$ terms are equal and together contribute
\begin{eqnarray*} 8 \left(1-\frac{\chi^{\tau}(\alpha)}{dim(\tau)}
\right) m_{\tau}(\tau^3) & = & 8
\left(1-\frac{\chi^{\tau}(\alpha)}{dim(\tau)} \right) \int_{g \in G}
\chi^{\tau}(g)^4\\ & = & 8 \left(1-\frac{\chi^{\tau}(\alpha)}{dim(\tau)}
\right) \int_{g \in G} \left| \sum_{\phi} m_{\phi}(\tau^2) \chi^{\phi}(g)
\right|^2\\ & = & 8 \left(1-\frac{\chi^{\tau}(\alpha)}{dim(\tau)} \right)
\sum_{\phi} m_{\phi}(\tau^2)^2. \end{eqnarray*} This completes the proof.
\end{proof}

    Putting the pieces together, one obtains the following
    theorem.

\begin{theorem}
 \label{CLTrealgroup} Let $G$ be a compact Lie group and let $\tau$ be
a non-trivial irreducible representation of $G$ whose character is real
valued. Fix a non-identity element $\alpha$ with the property that
$\alpha$ and $\alpha^{-1}$ are conjugate. Let $W=\chi^{\tau}(g)$ where $g$
is chosen from the Haar measure of $G$. Then for all real $x_0$,
\begin{eqnarray*} & & \left| \pp(W \leq x_0) - \frac{1}{\sqrt{2 \pi}}
\int_{-\infty}^{x_0} e^{-\frac{x^2}{2}} dx \right|
\\ & \leq & \sqrt{\sum_{\phi}^{\ \ *} m_{\phi}(\tau^2)^2
\left[ 2 - \frac{1}{a} \left(1 - \frac{\chi^{\phi}(\alpha)}{dim(\phi)}
\right) \right]^2}\\
& & + \left[ \frac{1}{\pi} \sum_{\phi} m_{\phi}(\tau^2)^2 \left(8 -
\frac{6}{a} \left(1-\frac{\chi^{\phi}(\alpha)}{dim(\phi)} \right) \right)
\right]^{1/4}. \end{eqnarray*} Here
$a=1-\frac{\chi^{\tau}(\alpha)}{dim(\tau)}$, the first sum is over all
non-trivial irreducible representations of $G$, and the second sum is over
all irreducible representations of $G$. \end{theorem}

\begin{proof} One applies Theorem \ref{steinbound} to the exchangeable pair
$(W,W')$ of this subsection. By Lemmas \ref{majorize} and \ref{condvar},
the first term in Theorem \ref{steinbound} gives the first term in the
theorem. To upper bound the second term in Theorem \ref{steinbound}, note
by the Cauchy-Schwartz inequality that \[ \ee|W'-W|^3 \leq
\sqrt{\ee(W'-W)^2 \ee(W'-W)^4}.\] Now use Lemma \ref{mom2} and part 2 of
Lemma \ref{highmom}. \end{proof}

\subsection{Example: $USp(2n,\cc)$} \label{Sp}

This subsection studies the distribution of $\chi^{\tau}(g)$, where $\tau$
is the $2n$ dimensional defining representation of $USp(2n,\cc)$. The only
representation theoretic fact needed is Lemma \ref{charsp}, which is the
$k=2$ case of a formula from page 200 of \cite{Su} giving the
decomposition of $\tau^k$ into irreducible representations. In its
statement, we let $x_1,x_1^{-1},\cdots,x_n,x_n^{-1}$ denote the
eigenvalues of an element of $USp(2n,\cc)$.

\begin{lemma} \label{charsp} For $n \geq 2$, the square of the defining
representation of the group $USp(2n,\cc)$ decomposes in a multiplicity
free way as the sum of the following three irreducible representations:
\begin{itemize}
\item The trivial representation, with character $1$
\item The representation with character $\frac{1}{2} (\sum_i x_i +
x_i^{-1})^2 + \frac{1}{2} \sum_i (x_i^2+x_i^{-2})$
\item The representation with character $\frac{1}{2} (\sum_i x_i +
x_i^{-1})^2 - \frac{1}{2} \sum_i (x_i^2+x_i^{-2}) - 1$
\end{itemize}
\end{lemma}

{\it Remark:} Lemma \ref{charsp} could also be easily guessed (and proved)
by looking at the character formulas for $USp(2n,\cc)$ on page 219 of
\cite{W}.

\begin{theorem} \label{Spthm}
Let $g$ be chosen from the Haar measure of $USp(2n,\cc)$, where $n \geq
2$. Let $W(g)$ be the trace of $g$. Then for all real $x_0$, \[ \left|
\pp(W \leq x_0) - \frac{1}{\sqrt{2 \pi}} \int_{-\infty}^{x_0}
e^{-\frac{x^2}{2}} dx \right| \leq \frac{\sqrt{2}}{n}.\]
\end{theorem}

\begin{proof} One applies Theorem \ref{CLTrealgroup}, with
 $\tau$ the defining representation, and $\alpha$ an element of type
 $\{x_1^{\pm 1},\cdots,x_n^{\pm n} \}$ where $x_1=\cdots =
x_{n-1}=1$ and $x_n=e^{i \theta}$. Then $\alpha$ is conjugate to
$\alpha^{-1}$ and one computes that $a=\frac{1-cos(\theta)}{n}$. Using
Lemma \ref{charsp}, one calculates that the first error term in Theorem
\ref{CLTrealgroup} is equal to $\frac{2
\sqrt{4cos(\theta)^2-4cos(\theta)+2}} {2n+1}$. One computes that the
second error term is equal to $\left[ \frac{24(1-cos(\theta))}{\pi (2n+1)}
\right]^{1/4}$. Since these bounds hold for all $\theta$ and are
continuous in $\theta$, the bounds hold in the limit that $\theta
\rightarrow 0$. This gives an upper bound of $\frac{2 \sqrt{2}}{2n+1} \leq
\frac{\sqrt{2}}{n}$, as claimed. \end{proof}

\subsection{\bf Example: $SO(2n+1,\rr)$} \label{SO}

We investigate the distribution of $\chi^{\tau}(g)$, where $\tau$ is the
$2n+1$-dimensional defining representation of $SO(2n+1,\rr)$. The only
ingredient from representation theory needed is Lemma \ref{charSOodd},
which is the $k=2$ case of a formula from page 204 of \cite{Su} giving the
decomposition of $\tau^k$ into irreducible representations (it is also
easily obtained by inspecting the character formulas on page 228 of
\cite{W}). In its statement, we let $x_1,x_1^{-1},\cdots,x_n,x_n^{-1},1$
be the eigenvalues of an element of $SO(2n+1,\rr)$.

\begin{lemma} \label{charSOodd} For $n \geq 2$, the square of the defining
representation of $SO(2n+1,\rr)$ decomposes in a multiplicity free way as
the sum of the following three irreducible representations:
\begin{itemize}
\item The trivial representation, with character $1$
\item The representation with character $\frac{1}{2} (\sum_i x_i +
x_i^{-1})^2 + \frac{1}{2} \sum_i (x_i^2+x_i^{-2}) + \sum_i (x_i+x_i^{-1})$
\item The representation with character $\frac{1}{2} (\sum_i x_i +
x_i^{-1})^2 - \frac{1}{2} \sum_i (x_i^2+x_i^{-2}) + \sum_i (x_i+x_i^{-1})$
\end{itemize}
\end{lemma}

    This leads to the following theorem.

\begin{theorem} \label{SOoddthm}
 Let $g$ be chosen from the Haar measure of $SO(2n+1,\rr)$, where $n \geq
 2$. Let $W(g)$ be the trace of $g$. Then for all real $x_0$, \[
 \left| \pp(W \leq x_0) - \frac{1}{\sqrt{2 \pi}} \int_{-\infty}^{x_0}
 e^{-\frac{x^2}{2}} dx \right| \leq \frac{\sqrt{2}}{n}.\]
\end{theorem}

\begin{proof} One applies Theorem \ref{CLTrealgroup}, taking
 $\tau$ to be the defining representation, and $\alpha$
 to be an element of type $\{x_1^{\pm 1},\cdots,x_n^{\pm n},1 \}$
 where $x_1=\cdots = x_{n-1}=1$ and $x_n=e^{i \theta}$ (i.e. $\alpha$ is a
 rotation by $\theta$). Then $\alpha$ is conjugate
 to $\alpha^{-1}$ and $a=\frac{2(1-cos(\theta))}{2n+1}$. Using Lemma \ref{charSOodd}
 one computes that in the $\theta \rightarrow 0$ limit the first error term in Theorem
 \ref{CLTrealgroup} is equal to $\frac{\sqrt{2}}{n}$. One calculates
 that the second error term is equal to $\left[ \frac{12(2n+1)(1-cos(\theta))}
 {\pi n (2n+3) } \right]^{1/4}$. The proof of the theorem is completed by
 noting that this goes to $0$ as $\theta \rightarrow 0$. \end{proof}

\subsection{Example: $O(2n,\rr)$} \label{O}

We consider the distribution of $\chi^{\tau}(g)$, where $\tau$ is the
$2n$-dimensional defining representation of $O(2n,\rr)$. The only
representation theoretic information needed is Lemma \ref{decomp}, which
is the $k=2$ case of a result of Proctor \cite{Pr} (and also not difficult
to obtain from the character formulas on page 228 of \cite{W}). In its
statement, we let $x_1,x_1^{-1},\cdots,x_n,x_n^{-1}$ be the eigenvalues of
an element of $O(2n,\rr)$.

\begin{lemma}
 \label{decomp} For $n \geq 2$, the square of the defining representation
 of $O(2n,\rr)$ decomposes in a multiplicity free way as the sum of the
 following three irreducible representations:
\begin{itemize}
\item The trivial representation, with character $1$
\item The representation with character $\frac{1}{2} \left( \sum_i
 x_i+ \overline{x_i} \right)^2 - \frac{1}{2} \sum_i \left( x_i^2
 + \overline{x_i}^2 \right)$
\item The representation with character $\frac{1}{2} \left( \sum_i
x_i+ \overline{x_i} \right)^2 + \frac{1}{2} \sum_i \left( x_i^2 +
\overline{x_i}^2 \right) -1$
\end{itemize}
\end{lemma}

    This leads to the following result.

\begin{theorem} \label{Oeventhm}  Let $g$ be chosen from the Haar measure of
$O(2n,\rr)$, where $n \geq 2$. Let $W(g)$ be the trace of $g$. Then for
all real $x_0$, \[ \left| \pp(W \leq x_0) - \frac{1}{\sqrt{2 \pi}}
\int_{-\infty}^{x_0}
 e^{-\frac{x^2}{2}} dx \right| \leq \frac{\sqrt{2}}{n-1}.\]
\end{theorem}

\begin{proof} Apply Theorem \ref{CLTrealgroup}, with
 $\tau$ the defining representation of $O(2n,\rr)$.
 We take $\alpha$ to be an element of type $\{x_1^{\pm
 1},\cdots,x_n^{\pm n} \}$ where $x_1=\cdots = x_{n-1}=1$ and
 $x_n=e^{i \theta}$ (i.e. $\alpha$ is a rotation by $\theta$).
 Then $\alpha$ is conjugate to $\alpha^{-1}$ and
 $a=\frac{1-cos(\theta)}{n}$. Lemma \ref{decomp} gives the decomposition
 of $\tau^2$ into irreducibles, and from this one calculates the first
 error term in Theorem \ref{CLTrealgroup}, and sees that in the
 $\theta \rightarrow 0$ limit it is equal to $\frac{\sqrt{8}}{2n-1}$.
 One computes that the second error term is equal to $\left[ \frac{24 n (1-cos(\theta))}
 {\pi (n+1) (2n-1) } \right]^{1/4}$. The proof of the theorem is completed by
 noting that this goes to $0$ as $\theta \rightarrow 0$. \end{proof}

\subsection{General theory (complex case)} \label{cgengroup1}

Let $G$ be a compact Lie group and $\tau$ be an irreducible representation
of $G$ whose character is not real valued. The random variable of interest
to us is $W= \frac{1}{\sqrt{2}} \left( \chi^{\tau}(g) +
\overline{\chi^{\tau}(g)} \right)$, where $g$ is chosen from the Haar
measure of $G$. It follows from the orthogonality relations for
irreducible characters of $G$ that $\ee(W)=0$ and $\ee(W^2)=1$.

We now define a pair $(W,W')$ by letting $W$ be as above and $W'=W(\alpha
g)$ where $\alpha$ is chosen uniformly at random from a fixed self-inverse
conjugacy class of $G$. As in Subsection \ref{gengroup1}, the pair
$(W,W')$ is exchangeable and all $\chi^{\phi}(\alpha)$ are real.

The remaining results in this subsection are proved by minor modifications
of the arguments in Subsection \ref{gengroup1}.

\begin{lemma} \label{cmom1} \[ \ee(W'|W) = \left( \frac{\chi^{\tau}(\alpha)}
{dim(\tau)} \right) W.\] \end{lemma}

\begin{lemma} \label{cmom2} \[ \ee(W'-W)^2 = 2 \left(1-  \frac{\chi^{\tau}
(\alpha)}{dim(\tau)} \right).\] \end{lemma}

\begin{lemma} \[ \ee[(W')^2|g]  = \frac{1}{2} \sum_{\phi}
m_{\phi}[(\tau+\overline{\tau})^2] \frac{\chi^{\phi}(\alpha)}{dim(\phi)}
\chi^{\phi}(g),\] where the sum is over all irreducible representations of
$G$.
\end{lemma}

\begin{lemma} \label{ccondvar} \[ Var([\ee(W'-W)^2|g]) = \frac{1}{4}
 \sum^{\ \ *}_{\phi} m_{\phi}[(\tau+\overline{\tau})^2]^2 \left(
 1 + \frac{\chi^{\phi}(\alpha)}{dim(\phi)}- \frac{2
 \chi^{\tau}(\alpha)}{dim(\tau)} \right)^2, \] where the star signifies
 that the sum is over all nontrivial irreducible representations of $G$.
\end{lemma}

\begin{lemma} \label{chighmom} Let $k$ be a positive integer.
\begin{enumerate}
\item $\ee(W'-W)^k$ is equal to \[ \frac{1}{2^{k/2}} \sum_{r=0}^k (-1)^{k-r} {k
\choose r} \sum_{\phi} m_{\phi}[(\tau+\overline{\tau})^r]
m_{\phi}[(\tau+\overline{\tau})^{k-r}]
\frac{\chi^{\phi}(\alpha)}{dim(\phi)}.\]

\item $\ee(W'-W)^4 = \sum_{\phi}
m_{\phi}[(\tau+\overline{\tau})^2]^2 \left[ 2 \left(
1-\frac{\chi^{\tau}(\alpha)}{dim(\tau)} \right) -\frac{3}{2} \left(1 -
\frac{\chi^{\phi}(\alpha)}{dim(\phi)} \right) \right]$.
\end{enumerate}
\end{lemma}

Finally, one obtains the following central limit theorem.

\begin{theorem}
 \label{CLTcomplexgroup} Let $G$ be a compact Lie group and $\tau$
an irreducible representation of $G$ whose character is not real valued.
Let $\alpha \neq 1$ be such that $\alpha$ and $\alpha^{-1}$ are conjugate.
Let $W=\frac{1}{\sqrt{2}} \left( \chi^{\tau}(g) +
\overline{\chi^{\tau}(g)} \right)$ where $g$ is chosen from the Haar
measure of $G$. Then for all real $x_0$,
\begin{eqnarray*} & & \left| \pp(W \leq x_0) - \frac{1}{\sqrt{2 \pi}}
\int_{-\infty}^{x_0} e^{-\frac{x^2}{2}} dx \right| \\ & \leq & \frac{1}{2}
\sqrt{\sum_{\phi}^{\ \ *} m_{\phi}[(\tau+\overline{\tau})^2]^2 \left[ 2-
\frac{1}{a} \left( 1-\frac{\chi^{\phi}(\alpha)}{dim(\phi)} \right)
\right]^2}\\
& & + \left[ \frac{1}{\pi} \sum_{\phi}
m_{\phi}[(\tau+\overline{\tau})^2]^2 \left(2 - \frac{3}{2a}
\left(1-\frac{\chi^{\phi}(\alpha)}{dim(\phi)} \right) \right)
\right]^{1/4}. \end{eqnarray*} Here
$a=1-\frac{\chi^{\tau}(\alpha)}{dim(\tau)}$, the first sum is over all
non-trivial irreducible representations of $G$, and the second sum is over
all irreducible representations of $G$. \end{theorem}

\subsection{Example: $U(n,\cc)$} \label{U}

    Every element of $U(n,\cc)$ is conjugate to a diagonal
    matrix with entries $(x_1,\cdots,x_n)$ and the representation
    theory of $U(n,\cc)$ is well understood (see for instance
    \cite{Bu}). The irreducible representations of
    $U(n,\cc)$ are parameterized by integer sequences $\lambda_1
    \geq \lambda_2 \geq \cdots \geq \lambda_n$. The corresponding character
    value on an element of type $\{x_1,\cdots,x_n\}$ is given by the Schur function
    $s_{\lambda}(x_1,\cdots,x_n)$. (The usual definition of Schur functions
    requires that $\lambda_n \geq 0$, so if $\lambda_n=-k<0$, this should be
    interpreted as $(x_1 \cdots x_n)^{-k} s_{\lambda+(k)^n}$, where
    $\lambda+(k)^n$ is given by adding $k$ to each of
    $\lambda_1,\cdots,\lambda_n$). The complex conjugate
    of a character with data $\lambda_1 \geq \lambda_2 \geq \cdots \geq \lambda_n$
    has data $-\lambda_n \geq -\lambda_{n-1} \geq \cdots \geq -\lambda_1$.

    Combining the above information with Theorem
    \ref{CLTcomplexgroup}, one obtains the following result.

\begin{theorem} \label{Uthm} Let $g$ be chosen from the Haar measure of
$U(n,\cc)$, where $n \geq 2$. Let $W(g) = \frac{1}{\sqrt{2}} [Tr(g) +
\overline{Tr(g)}]$ where $Tr$ denotes trace. Then for all real $x_0$, \[
\left| \pp(W \leq x_0) - \frac{1}{\sqrt{2 \pi}} \int_{-\infty}^{x_0}
e^{-\frac{x^2}{2}} dx \right| \leq \frac{2}{n-1}.\] \end{theorem}

\begin{proof} One applies Theorem \ref{CLTcomplexgroup} with $\tau$ the
$n$-dimensional defining representation. We take $\alpha$ to be an element
of type
 $\{x_1,\cdots,x_n \}$ with $x_1=\cdots=x_{n-2}=1$, $x_{n-1}=e^{i
 \theta}$, and $x_n=e^{- i \theta}$. Then $\alpha$ and $\alpha^{-1}$ are
 conjugate and $a=\frac{2(1-cos(\theta))}{n}$. By the Pieri rule for multiplying Schur
 functions (page 73 of \cite{Mac}), the decomposition of the character of
 $(\tau+\overline{\tau})^2$ in terms of Schur functions is given by
\begin{eqnarray*} (s_{(1)}+\overline{s_{(1)}})^2 & = & s_{(1)} s_{(1)} + 2
\frac{s_{(1)} s_{(1^{n-1})}}{x_1 \cdots x_n} + \overline{s_{(1)} s_{(1)}}
\\ & = & s_{(2)} + s_{(1,1)} + 2 \left[
\frac{s_{(1^n)}+s_{(2,1^{n-2})}} {x_1 \cdots x_n} \right] + s_{(-2)} +
s_{(-1,-1)}\\
& = & 2 + s_{(2)} + s_{(1,1)} + s_{(-2)} + s_{(-1,-1)} +
2s_{(1,0^{n-2},-1)}.\end{eqnarray*} One then computes that in the $\theta
\rightarrow 0$ limit, the first error
 term in Theorem \ref{CLTcomplexgroup} is equal to $\frac{2
 \sqrt{n^2+2}}{n^2-1} \leq \frac{2}{n-1}$. One computes that the
 second error term is equal to $
 \left[\frac{12(2n-1)(1-cos(\theta))}{\pi(n^2-1)} \right]^{1/4}$. The
 proof is completed by noting that this approaches 0 as $\theta
 \rightarrow 0$. \end{proof}

\section{Compact Symmetric Spaces} \label{sspace}

This section extends the methods of Section \ref{group} to study the
distribution of a fixed spherical function $\omega_{\tau}$ on a random
element of a compact symmetric space $G/K$. Subsection \ref{genss} gives
general theory for the case that $\omega_{\tau}$ is real valued. This is
illustrated for the sphere in Subsection \ref{sphere}, giving a different
perspective on a result of \cite{DF} and \cite{Me1}. We note that since
compact Lie groups can be viewed as symmetric spaces (see Section
\ref{genss}), the examples in Subsections \ref{Sp}, \ref{SO}, and \ref{O}
give further examples. Our theorems should also prove useful for
Jacobi-type ensembles arising from other root systems (see for instance
\cite{Vr}).

Subsection \ref{cgenss} indicates the changes needed to treat spherical
functions $\omega_{\tau}$ which are not real valued, and Subsections
\ref{circorth} and \ref{circsym} study the trace of elements from Dyson's
circular orthogonal and circular symplectic ensembles as special cases
(the circular unitary ensemble is equivalent to $U(n,\cc)$, so was already
treated in Subsection \ref{U}). Central limit theorems are known for the
trace of an element from the circular ensembles (see pages 48-9 of
\cite{Ra} and \cite{BF}), but our approach gives an error term.

\subsection{General theory (real case)} \label{genss}

To begin we recall some concepts about spherical functions of symmetric
spaces. Standard references which contain more details are \cite{He1},
\cite{He2}, \cite{Te}, and \cite{V}. Chapter 7 of \cite{Mac} is also very
helpful.

A Riemannian manifold $X$ is said to be a symmetric space if the geodesic
symmetry $\sigma:X \mapsto X$ with center at any point $x_0$ is an
isometry. Then $X$ can be identified with $G/K$, where $G$ is a connected
transitive Lie group of isometries of $X$, and $K$ is a compact group
which up to finite index is given by $K= \{ g \in G: gx_0=x_0 \}$.

A function $\omega_{\phi} \in L^2(G/K)$ is called spherical if
$\omega_{\phi}(1)=1$ and the following functional equation is satisfied:
\[ \int_K \omega_{\phi}(xky) dk = \omega_{\phi}(x) \omega_{\phi}(y)
 \ \forall x,y \in G.\] This equation implies that $\omega_{\phi}$ is $K$
 bi-invariant (i.e. $\omega_{\phi}(k_1 g k_2) = \omega_{\phi}(g)$ for all $k_1,k_2$
in $K$), which justifies our writing $\omega_{\phi}(g)$ instead of
$\omega_{\phi}(gK)$.

One reason that spherical functions are important is that if $G/K$ is
compact and $H_{\phi}$ is the $G$-invariant subspace of $L^2(G/K)$
generated by $\omega_{\phi}$, then $H_{\phi}$ is a finite dimensional
irreducible representation of $G$ and $L^2(G/K)$ is a direct sum of all
such $H_{\phi}$. We let $dim(\phi)$ denote the dimension of $H_{\phi}$.

In reading this section it is useful to keep in mind that a compact Lie
group $U$ can be viewed as a compact symmetric space. Indeed, one can take
$G=U \times U$ and $K$ the diagonal subgroup of $U$; then $G/K$ is
identified with $U$ under the mapping $(u_1,u_2)K \mapsto u_1u_2^{-1}$.
The spherical functions $\omega_{\phi}$ of $G/K$ are indexed by
irreducible representations $\phi$ of $U$ and are precisely the character
ratios $\frac{\chi^{\phi}(u)}{\chi^{\phi}(1)}$; moreover,
$dim(H_{\phi})=\chi^{\phi}(1)^2$.

Let $\omega_{\tau}$ be a non-trivial real valued spherical function of
$G/K$. We are interested in the distribution of $\omega_{\tau}(g)$
(normalized to have variance $1$). Here $gK$ is chosen from the ``Haar
measure'' $\mu$ on $G/K$, i.e. the unique $G$-invariant measure on $G/K$
which satisfies
\[ \int_G f(g) Haar(dg) = \int_{G/K} \left( \int_K f(gk) Haar(dk) \right)
\mu(dgK).\]

The following orthogonality relation will be used; see for instance page
45 of \cite{Kl} for a proof.

\begin{lemma} \label{orthogss}
\[ \int_{G/K} \omega_{\phi}(g) \overline{\omega_{\eta}(g)} =
\frac{\delta_{\phi,\eta}} {dim(\phi)}.\]
\end{lemma}

In particular, Lemma \ref{orthogss} implies that $W:= [dim(\tau)]^{1/2}
\omega_{\tau}$ has mean 0 and variance 1.

The following lemma is immediate from the functional equation for
$\omega_{\phi}$ and $K$ bi-invariance of $\omega_{\phi}$.

\begin{lemma} \label{ssdens} Let $G/K$ be a compact symmetric space,
and $\omega_{\phi}$ a spherical function of $G/K$. Then
\[ \int_{K \times K} \omega_{\phi}(k_1 \alpha k_2 g) dk_1 dk_2 =
\omega_{\phi}(\alpha) \omega_{\phi}(g) \] for all $\alpha,g \in G$.
\end{lemma}

We define the pair $(W,W')$ by letting $W=[dim(\tau)]^{1/2}
\omega_{\tau}(g)$ where $gK$ is from the ``Haar measure'' of $G/K$, and
$W'=W(\alpha g)$ where $\alpha$ is chosen uniformly at random from a fixed
double coset $K \alpha K \neq K$ which satisfies the property that $K
\alpha K = K \alpha^{-1}K$. Since $K \alpha^{-1} K = (K \alpha K)^{-1}$,
it follows that $(W,W')$ is exchangeable. Moreover the integral formula
for spherical functions (page 417 of \cite{He2}) implies that all
$\omega_{\phi}(\alpha)$ are real.

The analysis of the exchangeable pair $(W,W')$ can be carried out exactly
as in Subsection \ref{gengroup1}, using Lemmas \ref{orthogss} and
\ref{ssdens} instead of the orthogonality relations for compact Lie groups
and Lemma \ref{dens1}. Hence we simply record the results.

\begin{lemma} \label{ssmom1} \[ \ee(W'|W)= \omega_{\tau}(\alpha) W.\]
\end{lemma}

\begin{lemma} \label{ssmom2} \[ \ee(W'-W)^2 = 2(1-\omega_{\tau}(\alpha)).
\]
\end{lemma}

In the statements of the remaining results, we define the ``multiplicity''
$m_{\phi}(\tau^r)$ by the expansion \[ [\omega_{\tau}(g)]^r = \sum_{\phi}
\left[ \frac{dim(\phi)}{dim(\tau)^r} \right]^{1/2} m_{\phi}(\tau^r)
\omega_{\phi}(g).\] The numbers $m_{\phi}(\tau^r)$ are real and
non-negative (argue as on page 396 of \cite{Mac} with sums replaced by
integrals) but need not be integers. Note that by Lemma \ref{orthogss}, \[
m_{\phi}(\tau^r) = \left[ dim(\phi) dim(\tau)^r \right]^{1/2} \int_{G/K}
\omega_{\tau}(g)^r \overline{\omega_{\phi}(g)}.\]

\begin{lemma} \label{sscond2} \[ \ee[(W')^2|gK] = \sum_{\phi}
m_{\phi}(\tau)^2 [dim(\phi)]^{1/2} \omega_{\phi}(\alpha) \omega_{\phi}(g),
\] where the sum is over all spherical functions of $G/K$.
\end{lemma}

\begin{lemma} \label{sscondvar} \[ Var([\ee(W'-W)^2|gK]) = \sum_{\phi}^{\
\ *} m_{\phi}(\tau^2)^2 \left( 1+\omega_{\phi}(\alpha)-2
\omega_{\tau}(\alpha) \right)^2,\] where the star signifies that the sum
is over all nontrivial spherical functions of $G/K$.
\end{lemma}

\begin{lemma} \label{sshighmom} Let $k$ be a positive integer.
\begin{enumerate}
\item $\ee(W'-W)^k = \sum_{r=0}^k (-1)^{k-r} {k \choose r} \sum_{\phi}
m_{\phi}(\tau^r) m_{\phi}(\tau^{k-r}) \omega_{\phi}(\alpha)$.
\item $\ee(W'-W)^4 = \sum_{\phi} m_{\phi}(\tau^2)^2 \left[
8(1-\omega_{\tau}(\alpha)) - 6(1-\omega_{\phi}(\alpha)) \right]$.
\end{enumerate}
\end{lemma}

Finally, one has the following central limit theorem.

\begin{theorem} \label{ssCLTreal} Let $G/K$ be a compact symmetric space
and let $\omega_{\tau}$ be a non-trivial real-valued spherical function of
$G/K$. Fix an element $\alpha \not \in K$ such that $K \alpha K = K
\alpha^{-1} K$. Let $W= [dim(\tau)]^{1/2} \omega_{\tau}(g)$ where $gK$ is
chosen from the ``Haar measure'' of $G/K$. Then for all real $x_0$,
\begin{eqnarray*} & & \left| \pp(W \leq x_0) - \frac{1}{\sqrt{2 \pi}}
\int_{-\infty}^{x_0} e^{-\frac{x^2}{2}} dx \right| \\ & \leq &
\sqrt{\sum_{\phi}^{\ \ *} m_{\phi}(\tau^2)^2 \left[ 2 - \frac{1}{a} \left(
1 - \omega_{\phi}(\alpha) \right) \right]^2}\\
& & + \left[ \frac{1}{\pi} \sum_{\phi} m_{\phi}(\tau^2)^2 \left(8 -
\frac{6}{a} \left(1- \omega_{\phi}(\alpha) \right) \right) \right]^{1/4}.
\end{eqnarray*} Here $a=1-\omega_{\tau}(\alpha)$, the
first sum is over all non-trivial spherical functions of $G/K$, and the
second sum is over all spherical functions of $G/K$.
\end{theorem}

\subsection{Example: the sphere} \label{sphere} This subsection studies the unit
sphere in $\rr^n$, viewed as the symmetric space $SO(n,\rr)/SO(n-1,\rr)$.
Chapter 9 of \cite{V} is a good reference for the spherical functions of
this symmetric space, and Chapter 4 of \cite{DyM} is a very clear textbook
treatment for the special case $n=3$. Letting $e_1,\cdots,e_n$ be the
standard basis of $\rr^n$ and embedding $SO(n-1,\rr)$ inside $SO(n,\rr)$
as the subgroup fixing $e_n$, then $Kg_1K = Kg_2K$ if and only if
$g_1(e_n)$ and $g_2(e_n)$ have the same last coordinate. From this it is
not difficult to check that $KgK=Kg^{-1}K$ for all $g$, and that the
double cosets of $SO(n-1,\rr)$ in $SO(n,\rr)$ are parameterized by $x_n$,
the final coordinate of a point $(x_1,\cdots,x_n)$ on the sphere. In what
follows we let $x$ denote $x_n$.

From page 461 of \cite{V}, the spherical functions $\omega_{l}$ are
parameterized by integers $l \geq 0$ and satisfy \[ \omega_{l}(x) =
\frac{l! (n-3)!}{(l+n-3)!} C_l^{\frac{n-2}{2}}(x).\] Here $C_l^{\rho}$ is
the Gegenbauer polynomial, defined by the generating function \[ \sum_{l
\geq 0} C_l^{\rho}(x) t^l = (1-2xt+t^2)^{-\rho}.\] For instance,
\[ C_0^{\rho}(x)=1 , \ C_1^{\rho}(x) = 2 \rho x , \ C_2^{\rho}(x) = -\rho+
2\rho(1+\rho) x^2 \] and \[ \omega_{0}(x)=1, \ \omega_{1}(x)=x, \
\omega_{2}(x) = \frac{nx^2-1}{n-1}.\] By page 462 of \cite{V},
$dim(l)=\frac{(2l+n-2)(n+l-3)!}{(n-2)! l!}$.

We study the random variable $W(x)=[dim(1)]^{1/2} \omega_{1}= \sqrt{n} x$.
In fact sharp (up to constants) normal approximations for $W$ are known:
see Diaconis and Freedman \cite{DF} and also Meckes \cite{Me1}, who uses
Stein's method to obtain an error term of $\frac{2 \sqrt{3}}{n-1}$ in
total variation distance. Our viewpoint leads to the following result.

\begin{theorem} Let $W= \sqrt{n}x$, where $x$ is the last coordinate of
a random point on the unit sphere in $\rr^n$. Then for all real $x_0$, \[
\left| \pp(W \leq x_0) - \frac{1}{\sqrt{2 \pi}} \int_{-\infty}^{x_0}
e^{-\frac{x^2}{2}} dx \right| \leq \frac{2 \sqrt{2}}{n-1}.\]
\end{theorem}

\begin{proof} We apply Theorem \ref{ssCLTreal} with $\tau=\omega_1$.
Writing $\omega_1(x)^2$ as a linear combination of $\omega_0(x)$ and
$\omega_2(x)$, one computes that $m_{(0)}(\tau^2)=1, m_{(2)}(\tau^2) =
\sqrt{ \frac{2(n-1)}{n+2}}$, and that all other multiplicities in $\tau^2$
vanish. Letting $\alpha$ be less than one but close to 1, one computes
that the first error term in Theorem \ref{ssCLTreal} is exactly $\frac{2
\sqrt{2}}{n-1}$. Letting $\alpha$ tend to 1 from below, the second error
term in Theorem \ref{ssCLTreal} vanishes and the result follows.
\end{proof}

\subsection{General theory (complex case)} \label{cgenss} In this
subsection $G/K$ is a compact symmetric space and $\omega_{\tau}$ is a
spherical function which is not real valued. The random variable we study
is $W = \left[ \frac{dim(\tau)}{2} \right]^{1/2} \left( \omega_{\tau}(g)+
\overline{\omega_{\tau}(g)} \right)$. As in Subsection \ref{genss}, let
$W'=W(\alpha g)$ where $\alpha$ is chosen uniformly at random from a fixed
double coset $K \alpha K \neq K$ which satisfies the property that $K
\alpha K = K \alpha^{-1}K$. Then $(W,W')$ is exchangeable and all
$\omega_{\phi}(\alpha)$ are real.

The remaining results in this subsection extend those of Subsection
\ref{cgengroup1}, and are proved by nearly identical arguments, using
Lemmas \ref{orthogss} and \ref{ssdens}.

\begin{lemma} \label{cssmom1} $\ee(W'|W)= \omega_{\tau}(\alpha)W$.
\end{lemma}

\begin{lemma} \label{cssmom2} $\ee(W'-W)^2=2(1-\omega_{\tau}(\alpha))$.
\end{lemma}

For the remaining results in this subsection, we define the
``multiplicity'' $m_{\phi}[(\tau+\overline{\tau})^r]$ by the expansion
\[ (\omega_{\tau}+\overline{\omega_{\tau}})^r = \sum_{\phi} \left[
\frac{dim(\phi)}{dim(\tau)^r} \right]^{1/2}
m_{\phi}[(\tau+\overline{\tau})^r] \omega_{\phi}.\] Arguing as on page 396
of \cite{Mac}, one has that the numbers
$m_{\phi}[(\tau+\overline{\tau})^r]$ are real and non-negative (though not
necessarily integers). Note that by Lemma \ref{orthogss},
\[ m_{\phi}[(\tau+\overline{\tau})^r] = [dim(\phi) dim(\tau)^r]^{1/2}
\int_{G/K} \left( \omega_{\tau}(g)+\overline{\omega_{\tau}(g)} \right)^r
\overline{\omega_{\phi}(g)} .\]

\begin{lemma} \label{csscond2} \[ \ee[(W')^2|gK] = \frac{1}{2} \sum_{\phi}
m_{\phi}[(\tau+\overline{\tau})^2] dim(\phi)^{1/2} \omega_{\phi}(\alpha)
\omega_{\phi}(g), \] where the sum is over all spherical functions of
$G/K$.
\end{lemma}

\begin{lemma} \label{csscondvar} \[ Var([\ee(W'-W)^2|gK]) = \frac{1}{4} \sum_{\phi}^{ \
\ *} m_{\phi}[(\tau+\overline{\tau})^2]^2 (1+\omega_{\phi}(\alpha)-2
\omega_{\tau}(\alpha))^2, \] where the star signifies that the sum is over
all nontrivial spherical functions of $G/K$.
\end{lemma}

\begin{lemma} \label{csshighmom} Let $k$ be a positive integer.
\begin{enumerate}
\item $\ee(W'-W)^k$ is equal to \[ \frac{1}{2^{k/2}} \sum_{r=0}^k (-1)^{k-r} {k
\choose r} \sum_{\phi} m_{\phi}[(\tau+\overline{\tau})^r]
m_{\phi}[(\tau+\overline{\tau})^{k-r}] \omega_{\phi}(\alpha).\]
\item $\ee(W'-W)^4$ is equal to \[ \sum_{\phi} m_{\phi}[(\tau+\overline{\tau})^2]^2
\left[ 2(1-\omega_{\tau}(\alpha)) - \frac{3}{2}(1-\omega_{\phi}(\alpha))
\right].
\]
\end{enumerate}
\end{lemma}

Putting the pieces together, one has the following central limit theorem.

\begin{theorem} \label{cssCLT} Let $G/K$ be a compact symmetric space and
let $\omega_{\tau}$ be a spherical function of $G/K$ which is not real
valued. Fix an element $\alpha \not \in K$ such that $K \alpha K = K
\alpha^{-1}K$. Let $W = \sqrt{\frac{dim(\tau)}{2}} \left(
\omega_{\tau}(g)+ \overline{\omega_{\tau}(g)} \right)$, where $gK$ is from
the ``Haar measure'' of $G/K$. Then for all real $x_0$, \begin{eqnarray*}
& & \left| \pp(W \leq x_0) - \frac{1}{\sqrt{2 \pi}} \int_{-\infty}^{x_0}
e^{-\frac{x^2}{2}} dx \right| \\ & \leq & \frac{1}{2} \sqrt{\sum_{\phi}^{\
\ *} m_{\phi}[(\tau+\overline{\tau})^2]^2 \left[ 2 - \frac{1}{a} \left( 1-
\omega_{\phi}(\alpha) \right) \right]^2}\\
& & + \left[ \frac{1}{\pi} \sum_{\phi}
m_{\phi}[(\tau+\overline{\tau})^2]^2 \left(2 - \frac{3}{2a} \left(1-
\omega_{\phi}(\alpha) \right) \right) \right]^{1/4}.
\end{eqnarray*} Here $a=1-\omega_{\tau}(\alpha)$, the
first sum is over all non-trivial spherical functions of $G/K$, and the
second sum is over all spherical functions of $G/K$.
\end{theorem}

\subsection{Example: $U(n,\cc)/O(n,\rr)$} \label{circorth} This symmetric
space can be identified with the set of symmetric unitary matrices by the
map $g \mapsto gg^T$ (see \cite{Dn} or \cite{Fo} for details), and the
resulting matrix ensemble is known as Dyson's circular orthogonal
ensemble. For a thorough discussion of this ensemble, see the texts
\cite{Fo} or \cite{Mt}. In particular, it is known that if a function $f$
depends only on the eigenvalues $x_1,\cdots,x_n$ of a matrix from Dyson's
ensemble, then
\[ \int_{G/K}
f = \frac{[\Gamma(3/2)]^n}{\Gamma(\frac{n}{2}+1)} \int_{\mathbb{T}^n}
f(x_1,\cdots,x_n) \prod_{1 \leq i<j \leq n} |x_i-x_j| \prod_{k=1}^n
\frac{dx_k}{2 \pi}.\] In this integral, $\mathbb{T}^n$ is the
$n$-dimensional torus with coordinates \[ x_1, \cdots, x_n, \ x_i \in
\mathbb{C}, \ |x_i|=1. \] It is convenient to let the inner product
$\langle f,g \rangle$ of two functions be defined by \[ \langle f,g
\rangle = \frac{[\Gamma(3/2)]^n}{\Gamma(\frac{n}{2}+1)}
\int_{\mathbb{T}^n} f(x_1,\cdots,x_n) \overline{g(x_1,\cdots,x_n)}
\prod_{1 \leq i<j \leq n} |x_i-x_j| \prod_{k=1}^n \frac{dx_k}{2 \pi}\]

The spherical functions for this symmetric space are parameterized by
integer sequences $\lambda_1 \geq \lambda_2 \cdots \geq \lambda_n$ and are
$\omega_{\lambda}:= \frac{P_{\lambda}(x_1,\cdots,x_n;2)}
{P_{\lambda}(1,\cdots,1;2)}$, the normalized Jack polynomials with
parameter $2$. An excellent reference for Jack polynomials is Section 6.10
of \cite{Mac}. There one assumes that $\lambda_n \geq 0$, so if
$\lambda_n=-k<0$, $P_{\lambda}$ should be interpreted as $(x_1 \cdots
x_n)^{-k} P_{\lambda+(k)^n}$, where $\lambda+(k)^n$ is given by adding $k$
to each of $\lambda_1,\cdots,\lambda_n$.

To describe some useful combinatorial properties of Jack polynomials, we
use the notation that if $\lambda$ is a partition and $s$ is a box of
$\lambda$, then $l'(s),l(s),a(s),a'(s)$ are respectively the number of
squares in the diagram of $\lambda$ to the north of $s$ (in the same
column), south of $s$ (in the same column), east of $s$ (in the same row),
and west of $s$ (in the same row). For example the box marked $s$ in the
partition below 
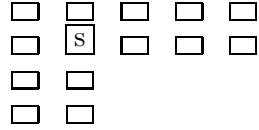
\[
\begin{array}{c c c c c} \framebox{\ }& \framebox{\ }& \framebox{\ } & \framebox{\ }
& \framebox{\ }  \\
\framebox{\ }& \framebox{s}& \framebox{\ } & \framebox{\ } & \framebox{\ } \\
\framebox{\ }& \framebox{\ }& & &\\ \framebox{\ }& \framebox{\ }& & &
\end{array}
\] would have $l'(s)=1$, $l(s)=2$, $a'(s)=1$, and $a(s)=3$.

Letting $\lambda$ be a partition of $n$, and using this notation, two
useful formulas are the ``principal specialization formula'' (page 381 of
\cite{Mac})
\[ P_{\lambda}(1,\cdots,1;2) = \prod_{s \in \lambda} \left[ \frac{n
+ 2a'(s)-l'(s)}{2a(s)+l(s)+1} \right]. \] and the formula
\[ dim(\lambda) = \prod_{s \in \lambda}
 \frac{\left( n+1+2a'(s)-l'(s) \right) \left( n+2a'(s)-l'(s) \right)}
{\left( 2 a(s)+l(s)+2 \right) \left( 2a(s)+l(s)+1 \right)},\] which
follows from the formula for $\langle P_{\lambda},P_{\lambda} \rangle$ on
page 383 of \cite{Mac} and the fact (Lemma \ref{orthogss}) that
\[ dim(\lambda) =
\frac{P_{\lambda}(1,\cdots,1;2)^2}{\langle P_{\lambda},P_{\lambda}
\rangle}.\]

\begin{theorem} Let $W= \frac{1}{2} \sqrt{1+\frac{1}{n}}
\left(Tr(g)+\overline{Tr(g)} \right)$, where $g$ is random from Dyson's
circular orthogonal ensemble, $Tr$ denotes trace, and $n \geq 2$. Then for
all real $x_0$,
\[ \left| \pp(W \leq x_0) - \frac{1}{\sqrt{2 \pi}} \int_{-\infty}^{x_0}
e^{-\frac{x^2}{2}} dx \right| \leq \frac{4}{n}.\] \end{theorem}

\begin{proof} Apply Theorem \ref{cssCLT} to the spherical function
$\tau=\omega_{(1)}(g)=\frac{Tr(g)}{n}$. To compute
$m_{\phi}[(\tau+\overline{\tau})^2]$ for all $\phi$, one has to decompose
$(\omega_{(1)}+\overline{\omega_{(1)}})^2$ into spherical functions, which
is equivalent to decomposing $(P_{(1)}+\overline{P_{(1)}})^2$ in terms of
Jack polynomials. From the Pieri rule for Jack polynomials (\cite{Mac},
page 340), one calculates that \begin{eqnarray*}
(P_{(1)}+\overline{P_{(1)}})^2 & = & P_{(1)} P_{(1)} + 2 \frac{P_{(1)}
P_{(1^{n-1})}}{x_1 \cdots x_n} + \overline{P_{(1)} P_{(1)}}\\
& = & P_{(2)} + \frac{4}{3} P_{(1^2)} + 2 \left[ \frac{ \frac{2n}{n+1}
P_{(1^n)} + P_{(2,1^{n-2})}}{x_1 \cdots x_n} \right] + \overline{P_{(2)}}
+ \frac{4}{3} \overline{P_{(1^2)}} \\
 & = &  \frac{4n}{n+1} + P_{(2)} + \frac{4}{3} P_{(1^2)} + 2
P_{(1,0^{n-2},-1)} + \overline{P_{(2)}} + \frac{4}{3} \overline{P_{(1^2)}}
.\end{eqnarray*}

We choose $\alpha$ to be an element of type $(1,\cdots,1,e^{i \theta},
e^{-i \theta})$. Then $K \alpha K = K \alpha^{-1}K$ and $a=
\frac{2(1-cos(\theta)}{n}$. One computes that in the $\theta \rightarrow
0$ limit, the first error term of Theorem \ref{cssCLT} is $\frac{1}{n}
\sqrt{\frac{8(n^3+2n^2+5n+6)}{n^3+4n^2+n-6}} \leq \frac{4}{n}$. The proof
is completed by computing that the second error term is $\left[
\frac{24(n+1)^2(1-cos(\theta))}{\pi n^2 (n+3)} \right]^{1/4}$ which goes
to $0$ as $\theta \rightarrow 0$.
\end{proof}

\subsection{Example: $U(2n,\cc)/USp(2n,\cc)$} \label{circsym} This
symmetric space corresponds to Dyson's circular symplectic ensemble
(\cite{Dn}); see \cite{Fo} or \cite{Mt} for background on this ensemble.
In particular, it is known that if $f$ depends only on the eigenvalues
$x_1,\cdots,x_n$ of a matrix from this ensemble, then
 \[ \int_{G/K} f = \frac{2^n}{(2n)!} \int_{\mathbb{T}^n} f(x_1,\cdots,x_n)
\prod_{1 \leq i<j \leq n} |x_i-x_j|^4 \prod_{k=1}^n \frac{dx_k}{2 \pi},\]
where $\mathbb{T}^n$ is as in the previous example. We let the inner
product $\langle f,g \rangle$ of two functions be defined by \[ \langle
f,g \rangle = \frac{2^n}{(2n)!} \int_{\mathbb{T}^n} f(x_1,\cdots,x_n)
\overline{g(x_1,\cdots,x_n)} \prod_{1 \leq i<j \leq n} |x_i-x_j|^4
\prod_{k=1}^n \frac{dx_k}{2 \pi}.\]

The spherical functions for this symmetric space are parameterized by
integer sequences $\lambda_1 \geq \lambda_2 \cdots \geq \lambda_n$ and are
$\omega_{\lambda}:= \frac{P_{\lambda}(x_1,\cdots,x_n;\frac{1}{2})}
{P_{\lambda}(1,\cdots,1;\frac{1}{2})}$, the normalized Jack polynomials
with parameter $1/2$. As mentioned in the previous example, Jack
polynomials are usually defined assuming that $\lambda_n \geq 0$, so if
$\lambda_n=-k<0$, $P_{\lambda}$ should be interpreted as $(x_1 \cdots
x_n)^{-k} P_{\lambda+(k)^n}$, where $\lambda+(k)^n$ is given by adding $k$
to each of $\lambda_1,\cdots,\lambda_n$.

Letting $\lambda$ be a partition of $n$ and using the notation of the
previous example, two useful formulas are the ``principal specialization
formula'' (page 381 of \cite{Mac})
\[ P_{\lambda}(1,\cdots,1;\frac{1}{2}) = \prod_{s \in \lambda} \left[ \frac{n
+ \frac{a'(s)}{2}-l'(s)}{\frac{a(s)}{2}+l(s)+1} \right]. \] and the
formula
\[ dim(H_{\lambda}) = \prod_{s \in \lambda}
 \frac{\left( n+\frac{a'(s)}{2}-l'(s) \right) \left( 2n-1+a'(s)-2l'(s) \right)}
{\left( \frac{a(s)}{2}+l(s)+1 \right) \left( a(s)+2l(s)+1 \right)},\]
which follows from the formula for $\langle P_{\lambda},P_{\lambda}
\rangle$ on page 383 of \cite{Mac} and the fact (Lemma \ref{orthogss})
that
\[ dim(\lambda) = \frac{P_{\lambda}(1,\cdots,1;\frac{1}{2})^2}{\langle
P_{\lambda},P_{\lambda} \rangle}.\]

\begin{theorem} Let $W= \sqrt{1-\frac{1}{2n}}
\left(Tr(g)+\overline{Tr(g)} \right)$, where $g$ is random from Dyson's
circular symplectic ensemble and $n \geq 2$. Then for all real $x_0$, \[
\left| \pp(W \leq x_0) - \frac{1}{\sqrt{2 \pi}} \int_{-\infty}^{x_0}
e^{-\frac{x^2}{2}} dx \right| \leq \frac{4}{n}.\] \end{theorem}

\begin{proof} Apply Theorem \ref{cssCLT} to the spherical function
$\tau=\omega_{(1)}(g)=\frac{Tr(g)}{n}$. To compute
$m_{\phi}[(\tau+\overline{\tau})^2]$ for all $\phi$, one has to decompose
$(\omega_{(1)}+\overline{\omega_{(1)}})^2$ into spherical functions, which
is equivalent to decomposing $(P_{(1)}+\overline{P_{(1)}})^2$ in terms of
Jack polynomials. From the Pieri rule for Jack polynomials (\cite{Mac},
page 340), one calculates that $(P_{(1)}+\overline{P_{(1)}})^2$ is equal
to \begin{eqnarray*}
 & & P_{(1)} P_{(1)} + 2 \frac{P_{(1)}
P_{(1^{n-1})}}{x_1 \cdots x_n} + \overline{P_{(1)} P_{(1)}}\\
& = & P_{(2)} + \frac{2}{3} P_{(1^2)} + 2 \left[ \frac{ \frac{n}{2n-1}
P_{(1^n)} + P_{(2,1^{n-2})}}{x_1 \cdots x_n} \right] + \frac{2}{3}
\overline{P_{(1^2)}} + \overline{P_{(2)}} \\
& = & \frac{2n}{2n-1} + P_{(2)} + \frac{2}{3} P_{(1^2)} + 2
P_{(1,0^{n-2},-1)} + \frac{2}{3} \overline{P_{(1^2)}} +
\overline{P_{(2)}}.\end{eqnarray*}

Now take $\alpha$ to be an element of type $(1,\cdots,1,e^{i \theta},
e^{-i \theta})$; then $K \alpha K = K \alpha^{-1}K$ and $a=
\frac{2(1-cos(\theta)}{n}$. One computes that in the $\theta \rightarrow
0$ limit, the first error term of Theorem \ref{cssCLT} is $\frac{1}{2n}
\sqrt{\frac{8(4n^3-4n^2+5n-3)}{4n^3-8n^2+n+3}} \leq \frac{4}{n}$. The
second error term is computed to be $\left[
\frac{6(2n-1)(4n-5)(1-cos(\theta))}{\pi n^2 (2n-3)} \right]^{1/4}$ which
goes to $0$ as $\theta \rightarrow 0$. \end{proof}

\end{document}